\documentclass[12pt]{amsart}
\usepackage{amsmath, amsthm, amscd, amsfonts, amssymb, graphicx, color}
\usepackage[bookmarksnumbered, colorlinks, plainpages]{hyperref}

\usepackage{graphicx}
\usepackage{caption}
\usepackage{pgf,tikz,pgfplots}
\usepackage{mathrsfs}
\usetikzlibrary{arrows}

\definecolor{ududff}{rgb}{0.30196078431372547,0.30196078431372547,1.}
\definecolor{xdxdff}{rgb}{0.49019607843137253,0.49019607843137253,1.}

\newtheorem{theorem}{Theorem}[section]
\newtheorem{proposition}[theorem]{Proposition}

\newtheorem{corollary}[theorem]{Corollary}

\newtheorem{lemma}[theorem]{Lemma}
\theoremstyle{definition}
\newtheorem{definition}[theorem]{Definition}

\newtheorem{remark}[theorem]{Remark}

\begin{document}
\title[minimal  Partition-Free groups ]{minimal  Partition-Free groups }
\author[A. Bahri, Z. Akhlaghi, B. Khosravi]{Afsane Bahri \& Zeinab Akhlaghi \& Behrooz Khosravi }
\address{ Dept. of Pure  Math.,  Faculty  of Math. and Computer Sci. \\
Amirkabir University of Technology (Tehran Polytechnic)\\ 424,
Hafez Ave., Tehran 15914, Iran \newline }
\email{afsanebahri@aut.ac.ir}
\email{z$\_$akhlaghi@aut.ac.ir}
\email{khosravibbb@yahoo.com}

\thanks{}
\subjclass[2010]{20C15, 20D05, 20D60.}

\keywords{Partition of group, classification of groups}

\begin{abstract}
Let $G$ be a finite group. A collection $\Pi=\{H_1,\dots,{H_r}\}$ of subgroups of $G$, where $r>1$, is said a non-trivial partition of $G$ if every non-identity element of $G$ belongs to one and only one $H_i$, for some $1\leqslant i\leqslant r$. We call a group $G$  that does not admit any non-trivial partition {\it a partition-free group}. In this paper, we study a partition-free group $G$ whose all proper non-cyclic subgroups admit non-trivial partitions. 
\keywords{Classification of groups\and Partitions of Groups}
\end{abstract}

\maketitle
\section{ Introduction}
only one $H_i$, for some $1\leqslant i\leqslant r$.  Obviously, cyclic groups do not admit a  non-trivial partition.  We call a non-cyclic group that does not admit a non-trivial  partition, a {\it partition-free group} 
or briefly a {\it PF-group}. We also call a subgroup $H$ of $G$ a {\it PF-subgroup}, if $H$ does not admit a non-trivial partition. Let $\Omega$ be a class of groups. A group is said minimal non-$\Omega$ if it is not an $\Omega$-group, while all its proper subgroups belong to $\Omega$.  Associating different properties to $\Omega$ and studying the structure of   minimal non-$\Omega$ groups have attracted the attention of many researchers.  One of the application of  studying such a groups  are using them in  arguments applying induction  on finite groups.

Minimal non-$\Omega$ groups studied for various of classes of $\Omega$,  for instance,  minimal non-cyclic and non-abelian groups are  analyzed in \cite{minimalcyclic}. Also, well-known Schmit 
theorem classifies minimal non-nilpotent groups. Moreover, minimal non-solvable  groups are determined in \cite{thompson}, and called minimal simple groups.  We  call $G$ a {\it minimal  PF-group}, if $G$ is a PF-group whose all proper non-cyclic subgroups admit non-trivial partitions.

Throughout the proofs, we come across special structure of a Frobenius groups several times and for more convenient we call them  minimal Frobenius groups, that are Frobenius groups  whose none of proper subgroups are Frobenius groups. It is easy to see that in such a group the Frobenius kernel is the only  minimal normal  subgroup of the group and Frobenius complements are groups of prime order. Our main result is:

\textbf{Main Theorem} Suppose that $G$ is a finite group. Then, $G$ is a minimal PF-group if and only if it is isomorphic to one of the following groups:
\begin{enumerate}
	\item[(1)] $Q_8$, the quaternion group of order $8$; 
	\item[(2)] $C_p\times C_{p^2}$, where $p$ is a prime;
	\item[(3)] $M_3(p)$, where $p\not =2$; 
	\item[(4)] $C_p^2\times C_q$, where $q$ and $p$ are  distinct primes;
	\item[(5)]A quasi-Frobenius group isomorphic to  $C_r\times (C_p: C_q)$, where $q,p$ and $r$ are distinct primes;
	\item[(6)] $\langle{x,y |\,\, x^q=y^{p^{n}}=1,\, y^{-1}xy=x^{r}}\rangle$, where $p,q$ are distinct primes, $n\geqslant 2$ a natural number,  $r\equiv1\mod{p}$, and  $r^p\equiv1\mod{q}$;
	\item[(7)] A quasi-Frobenius group $G$ in which $G/{\bf Z}(G)$ is a minimal Frobenius group of order $q^{\alpha}p$ for some primes $p,q$ and positive integer $\alpha$. Moreover, ${\bf Z} (G)\cong C_q$ and the Sylow $q$-subgroup of $G$ is a group of exponent $q$.

\end{enumerate}

\smallskip

Indeed, groups having a non-trivial partition are classified  (see \cite{baer1,baer2,kegel,suzuki}) completely, and it is the critical key to  prove our main result.  
We refer the readers to \cite{zappa} for more details.   

To pursue the main result, using CFSG and Thompson's classification of minimal simple groups, we show that minimal PF-groups are solvable. Beside the classification of groups admitting non-trivial partition, we apply the classification of finite $p$-groups with a cyclic maximal subgroup to prove the main theorem, recurrently.

Throughout the  paper, all groups are finite. By  $G^k$ we mean the direct product of $k$ copies of $G$, where $k$ is a natural number. In addition,  we  denote by $M:N$  the semidirect product of the groups $M$ by $N$ where the  action of $N$ on $M$ is certainly non-trivial and we use the notation $M\rtimes N$  for the semidirect product of $M$ by $N$  in general (whether the action of $N$ on $M$ is trivial or not).
If also $n $ is a natural number, we denote by $n_p$ the $p$-part of $n$, that is, $p^a$ if $p^a\mid n$ and $p^{a+1}$ does not divide $n$.      For the rest of notation, we follow \cite{atlas}.

\section{Preliminaries}

One of the classes of groups admitting non-trivial partitions is groups of Hughes-Thompson type which are defined as following:

\begin{definition} \rm(\cite{zappa}\rm)
	Let $G$ be a group and $p$ a prime divisor of $|G|$. The subgroup generated by all elements of $G$ whose order is not $p$ is said the {\it Hughes subgroup} and denoted by $H_p(G)$. The group $G$ is also said a group of {\it Hughes-Thompson type} if it is not a $p$-group and $H_p(G)<G$,  for some prime $p$.  
\end{definition}
It is proved that if $G$ is a group of Hughes-Thompson type, then $H_p(G)$ is nilpotent and $|G:H_p(G)|=p$ (see \cite{zappa}). Thus, $G\cong H_p(G): C_p$. As a consequence of these results, we have the following  Remark:

\begin{remark}\label{smallresults} Let G be a finite group.
	\item[(\textit{i})] Let $G$ be a group of Hughes-Thompson type isomorphic to $H_p(G): P$,  where $P\cong C_p$. Then, ${\bf{C}}_{H_p(G)}(P)$ is trivial or a $p$-group of exponent $p$;
	
	\item[(\textit{ii})] Suppose that $G\cong R: T$,  where $T$ is cyclic and $(|R|,|T|)=1$. Then,  $G$ is a group of Hughes-Thompson type if and only if  $T$ acts Frobeniusly on $R$, $R=H_p(G)$, and $|T|=p$ for some prime $p$;
	
	\item[(\textit{iii})] Let $G$ be a group of Hughes-Thompson type, where $H_p(G)< G$ for some prime $p$. Assume also that $N\unlhd G$. Then, $\overline{G}=G/N$ is a group  with $H_p(\overline{G})<   \overline{G}$. In addition, if $N\not\leq H_p(G)$, then  $\overline{G}$ is a $p$-group of exponent $p$;
	
	\item[(\textit{iv})] Let $G\cong M_1\times M_2$ in which $M_1$ and $M_2$ are non-trivial groups. Then, $G$ is  a group of  Hughes-Thompson type if and only if   for some $i\in \{1,2\}$, $M_i$ is a group of Hughes-Thompson type with $H_p(M_i)< M_i$ for some  prime  $p$ dividing $|M_i|$, and for $i\not= j\in \{1,2\}$,  $M_j$ is a $p$-group of exponent $p$.
	

\end{remark}
\begin{proof}
	
	(\textit{i}) If $P$ acts Frobeniusly on $H_p(G)$, then ${\bf{C}}_{H_p(G)}(P)$ is trivial. Hence, assume that $P=\langle x \rangle$ does not act Frobeniusly on $H_p(G)$. Suppose also that there exists an element $y\in {\bf{C}}_{H_p(G)}(P)$ of order $t\not =p$. Thus, by the definition of $H_p(G)$, $xy$ belongs to $H_p(G)$, and so $x\in H_p(G)$, a contradiction. Therefore, ${\bf{C}}_{H_p(G)}(P)$ is trivial or a $p$-group of exponent $p$. 
	
	\item[(\textit{ii})]  As  the "If part " of our claim easily can be concluded,  we concentrate  on the "only if  part". Suppose that $G$ is a group of Hughes-Thompson type with $H_p(G)<G$, where $p$ is a divisor of $|G|$. If $p\mid |R|$, then by the nilpotency (see \cite{zappa}) and normality of $H_p(G)$, $T\unlhd G$, and so $G\cong R\times T$. By (i), $p$ is a divisor of $|T|$, a contradiction. Hence, $p$ is a divisor of $|T|$, and so $R\leq H_p(G)$. Note that $T$ is cyclic, and so $|T|=p$ since otherwise the generator of $T$ belongs to $H_p(G)$. Consequently, $R=H_p(G)$. Moreover, by (i), $T$ acts Frobeniusly on $R$. 
	\item[(\textit{iii})] Assume that $G\cong H_p(G): P$ is a group of Hughes-Thompson type, where $P\cong C_p$. If $N\leq H_p(G)$, then we can see that $H_p(\overline{G})< \overline{G}$, easily.  Now, assume that $N$ contains an element $x\in G\setminus H_p(G)$.  Let $H=N\cap H_p(G)$. Then, $N\cong H: \langle x\rangle$. Without loss of generality, we may assume that  $P=\langle x \rangle$. Therefore, $G/H\cong H_p(G)/H \times PH/H$.  If there exists a prime $q\neq p$ such that $q\mid |H_p(G):H|$, then there exists an element $yH\in G/H$ of order $pq$ such that $y^{q}H=xH$. Hence, $y^{q}=xh$, where $h\in H$.  Noting  that by the definition of $H_p(G)$, $y$ belongs to $H_p(G)$, and so  $x\in H_p(G)$, a contradiction. Consequently, $p$ is the only prime divisor of $|H_p(G):H|$. Now, suppose that  there exists an element $yH\in H_p(G)/H$ whose order is at least $p^2$. Thus, $o(xyH)\geqslant p^2$, and so $o(xy)\geqslant p^2$, implying that $x\in H_p(G)$, a contradiction. Therefore, $\overline{G}$ is a $p$-group of exponent $p$ and so $H_p(\overline{G})=1$.

	\item[(\textit{iv})] Assume that $M_1$ is a group of Hughes-Thompson Type, $H_p(M_1)<M_1$, and $M_2$ is a group of exponent $p$. Then, $H_p(G)=H_p(M_1)\times M_2< G$, implying that $G$ is a group of Hughes-Thompson type. Conversely,   assume that $G$ is a group of Hughes-Thompson type, and so  $H_p(G)<G$ for some prime $p$. Without loss of generality, we may assume that $M_1$ is a subgroup of $H_p(G)$ and $M_2$ is not. Then, by (iii), $M_2$ is a $p$-group of exponent $p$ and $H_p(M_1)< M_1$. Since $G$ is not a $p$-group, $M_1$ is not a $p$-group. Therefore, $M_1$ is a group of Hughes-Thompson type, as desired. 
	
\end{proof}

Baer, Kegel, and Suzuki proved that:

\begin{lemma}\label{classification} \rm(\cite{baer1,baer2,kegel,suzuki}\rm) A finite group $G$ admits a non-trivial partition if and only if it is isomorphic to one of the following groups.
	\begin{enumerate}
		\item[(1)] A $p$-group with $H_p(G)\neq{G}$ and $|G|>p$;
		\item[(2)] A Frobenius group;
		\item[(3)] A group of Hughes-Thompson type;
		\item[(4)] ${\rm PGL}(2,p^h)$, where $p$ is an odd prime;
		\item[(5)] ${\rm PSL}(2,p^h)$, where $p$ is prime;
		\item[(6)] ${\rm Sz}(2^h)$, the Suzuki group of order $2^h$.
	\end{enumerate}
	In all the above cases, $h$ is a natural number.
\end{lemma}
Therefore, if $G$ is a solvable group admitting a non-trivial partition, then either $G$ satisfies one of the parts (1)-(3) of Lemma \ref{classification}, or $G\cong $PGL$(2,3)\cong S_4$. We also note that among given groups in the above lemma,  the only nilpotent  groups are the ones  satisfy part (1).  Hence, the only non-solvable groups admitting non-trivial partitions are  the almost simple groups of types ${\rm PGL}(2,p^h)$,  ${\rm PSL}(2,p^h)$ and ${\rm Sz}(2^h)$. We use theses results throughout the paper several times.    


\begin{lemma}\label{p-groups}\rm(\cite[Theorem 5.3.4]{robinson}\rm) 
	A group of order $p^n$ has a cyclic maximal subgroup if and only if it is one of the following types:
	\begin{enumerate}
		\item[(1)] $C_{p^n}$;
		\item[(2)] $C_p\times{C_{p^{n-1}}}$;
		\item[(3)] $M_n(p)=\langle{x,y |\,\, x^p=y^{p^{n-1}}=1,\, x^{-1}yx=y^{1+p^{n-2}}}\rangle$ with $n\geqslant 3$;
		\item[(4)] $Q_{2^n}\cong \langle{x,y |\,\, x^{2^{n-1}}=1, x^{2^{n-2}}=y^{2},\,y^{-1}xy=x^{-1}\rangle}$ with $n\geqslant 3$;
		\item[(5)] $D_{2^n}\cong \langle{x,y |\,\, x^{2^{n-1}}=y^{2}=1, \,y^{-1}xy=x^{-1}\rangle}$ with $n\geqslant 3$;
		\item[(6)] $SD_{2^n}\cong \langle{x,y |\,\,x^{2^{n-1}}=y^{2}=1,\,y^{-1}xy=x^{2^{n-2}-1}\rangle}$ with $n\geqslant 3$.
	\end{enumerate}
\end{lemma}
\begin{corollary}\label{smallresultse}
	Let $G$ be a finite group. 
	\item[(\textit{i})] If $G$ is an abelian group, admitting a non-trivial partition, then $G$ is an elementary abelian $p$-group, where $p$ is a prime.
	
	\item[(\textit{ii})]  Suppose that  $G$ is  a $p$-group,  containing a cyclic maximal subgroup.  Then, $G$ admits a non-trivial partition if and only if  $G\cong C_p^2$ or $D_{2^n}$, for some $n\geqslant 3$.

\end{corollary}
\begin{proof}
	(\textit{i}) See \cite[p.1]{abeliangroups}.
	\item[(\textit{ii})] Firstly, we show that the groups isomorphic to $C_p\times C_p$ or the dihedral group of order $2^n$, for some natural number $n\geqslant 3$, admit non-trivial partitions. By (i), the group $C_p\times C_p$ admits a non-trivial partition. Moreover, $H_2(D_{2^n})$ is a normal subgroup of $D_{2^n}$ of order $2^{n-1}$. Hence, by Lemma \ref{classification}, $D_{2^n}$ admits a non-trivial partition.
	
	Conversely, to prove the theorem, we just need to discuss about non-cyclic groups. In the case $G\cong C_p\times C_p^{n-1}$, by (i), we get that if $G$ admits a non-trivial partition, then $G\cong C_p^2$, as wanted. For the remaining non-cyclic $p$-groups $G$ satisfying the hypotheses of Lemma \ref{p-groups}, except the dihedral group of order $2^n$, we show that $G=H_p(G)$. If $G\cong M_n(p)$, then $o(y), o(xy)\geqslant p^2$, and so $x,y\in H_p(G)$. In the case $G\cong Q_{2^n}$, since $o(x), o(y)\geqslant 4$, we have  $x,y\in H_2(G)$. In addition, for $G\cong {SD}_{2^n}$, $o(x), o(xy)\geqslant 4$, and so $x,y\in H_2(G)$. Hence, in the above cases, $G=H_p(G)$.

\end{proof}


\section{\bf{  Some Examples of minimal PF-groups}}
Here, we study some groups we deal with in the next sections. And showing that why these groups are minimal PF-groups, that is, the non-cyclic groups not admitting a non-trivial partition, but all of its  proper non-cyclic subgroups admit non-trivial partitions. 

\bigskip

(\textit{1}) The first group is the quaternion group of order $8$. By Corollary \ref{smallresultse}(ii), it does not admit a non-trivial partition. Recall that all proper subgroups of $Q_8$ are cyclic, and so $Q_8$ is a minimal PF-group.

\bigskip

(\textit{2}) $C_p\times C_{p^2}$, where $p$ is a prime. By Corollary \ref{smallresultse}(i), the group dose not admit a non-trivial partition. We also note that proper subgroups of the group are cyclic or isomorphic to $C_p\times C_p$ that by Corollary \ref{smallresultse}(i), the subgroups isomorphic to $C_p\times C_p$ admit non-trivial partitions, as  desired.

\bigskip


(\textit{3}) $M_3(p)$ with $p\neq 2$. By Corollary \ref{smallresultse}(ii), $M_3(p)$, where $p\neq 2$, does not admit a non-trivial partition. Note that proper subgroups of $M_3(p)$ are cyclic or isomorphic to $C_p\times C_p$, which is discussed in the previous case. Therefore, $M_3(p)$, with $p\neq 2$, is a minimal PF-group.  

\bigskip


\bigskip

(\textit{4}) $G\cong C_p^2\times C_q$, where $p$ and $q$ are distinct primes. Corollary \ref{smallresultse}(i) implies that the group does not admit a non-trivial partition, while the only  proper non-cyclic subgroup of $G$ admits a  non-trivial partition. Thus, $G$ is a minimal PF-group.

\bigskip

(\textit{5}) The quasi-Frobenius group isomorphic to $C_r\times (C_p: C_q)$, where $r,p$ and $q$ are distinct primes. Since the group is solvable, if it admits a non-trivial partition, according to the expressed  argument following Lemma  \ref{classification}, it must be a group of Hughes-Thompson type. Remark \ref{smallresults}(ii) and Remark \ref{smallresults}(iv) implies that $r=q$ which is impossible. Note that the only  proper non-cyclic subgroup of $G$ is a  Frobenius group, and  recalling  Lemma \ref{classification}, this subgroup admits  a non-trivial partition. Thus, the quasi-Frobenius group isomorphic to $C_r\times (C_p: C_q)$, where $r,p$ and $q$ are distinct primes, is a minimal PF-group.

\bigskip

(\textit{6}) The group $G$ isomorphic to $\langle{x,y |\,\, x^q=y^{p^{n}}=1,\, y^{-1}xy=x^{r}}\rangle$, where $p,q$ are distinct primes, $n\geqslant 2$ a natural number,  $r\equiv1 \mod{p}$, and  $r^p\equiv1\mod{q}$. 
We know that $G\cong Q: P$, where $Q\in {\rm Syl}_q(G)$ and $P\in {\rm Syl}_p(G)$. Also, ${\bf{Z}}(G)\neq 1$ since $[x,y^p]=1$. By \cite{minimalcyclic}, $G$ is a minimal non-cyclic group, and so it is enough to show that $G$ is a PF-group.
Suppose that $G$ admits a non-trivial partition.  Since $G$ is solvable, by the expressed argument following Lemma \ref{classification}, it must be a group of Hughes-Thompson type. Remark \ref{smallresults}(ii) leads us to the fact that $G$ is a Frobenius group which is impossible, as ${\bf{Z}}(G)\neq 1$. Hence, $G$ is a minimal PF-group.

\bigskip

(\textit{7}) A quasi-Frobenius group $G$ in which $G/{\bf Z}(G)$ is a minimal Frobenius group of order $q^{\alpha}p$ for some primes $p,q$ and positive integer $\alpha$, also  ${\bf Z} (G)\cong C_q$ and the Sylow $q$-subgroup of $G$ is a group of exponent $q$. We know that, $G\cong Q: P$, where $Q\in {\rm Syl}_q(G)$ and $P\in {\rm Syl}_p(G)$. Similar to the above discussion, if $G$ admits a non-trivial partition, then it must be a group of Hughes-Thompson type (not a Frobenius group) which is a contradiction by Remark \ref{smallresults}(ii). Hence, $G$ is a PF-group. Note that since $Q$ is a group of exponent $q$, $H_q(Q)=1$, and so by Lemma \ref{classification}, all  the non-cyclic $q$-subgroups of $G$ admit  non-trivial partitions. It is remain to investigate the  proper subgroups of $G$ whose order is divided by  $pq$. Suppose that $H$ is  such a subgroup and    $Q_H\in {\rm Syl}_q(H)$. First, assume that ${\bf{Z}}(G)\nleq Q_{H}$, then $H$ is a Frobenius group and by Lemma \ref{classification}, it admits a non-trivial partition. Now, suppose that ${\bf{Z}}(G)\leq Q_{H}$. Without loss of generality, we may assume that $P\in $ Syl$_p(H)$. Then,  $P\leq N_G(Q_H)$. Note that $Q/{\bf{Z}}(G)$ is  an elementary abelian group, implying that  $Q\leq N_G(Q_H)$. Hence, $Q_H\unlhd G$. But $Q/{\bf{Z}}(G)$ is a minimal normal subgroup of $G/{\bf Z}(G)$, implying that $Q_H=Q$, and so $G=H$, a contradiction. Therefore, non-cyclic proper subgroups of $G$ admit non-trivial partitions, and so $G$ is a minimal PF-group.



\section{\bf{Solvability of  minimal PF-groups}}
Our main goal in this section is to prove the solvability of minimal PF-groups. 

\begin{lemma}\label{almostsimple}
	Let $G$ be an almost simple group, with socle $S$ which is isomorphic to either {\rm Sz}$(2^\alpha)$ or {\rm PSL}$(2,p^\alpha)$, where $p$ is a prime and $\alpha\geqslant 1$ an integer. Then, $G$ is not a minimal PF-group.
\end{lemma}

\begin{proof}
	On the contrary, suppose that $G$ is a minimal PF-group.
	By Lemma \ref{classification}, ${\rm Sz}(2^\alpha)$ and ${\rm PSL}(2,p^\alpha)$ are not minimal PF-groups, and so we may assume that  $S< G\leq $Aut$(S)$.   Firstly, suppose that $S={\rm Sz}(2^{\alpha})$. Therefore, $G$ is an extension of $S$ by $\langle \phi\rangle$, where $\phi$ is a non-trivial field automorphism of $S$. Thus, $H\cong {\rm Sz}(2)\times\langle\phi\rangle$ is a proper subgroup of $G$. Note that Sz$(2)\cong C_5: C_4$ is a Frobenius group. Looking at the  classification of solvable groups admitting a non-trivial partition and recalling  Remark \ref{smallresults}(ii) and Remark \ref{smallresults}(iv),  we see that $H$ does not admit a non-trivial partition, a contradiction. 
	
	Now, let $S={\rm PSL}(2,p^{\alpha})$. First,  assume that $G/S$ contains a non-trivial field automorphism $\phi$, then by Lemma \ref{classification} and Remark \ref{smallresults}(iv), $H\cong {\rm PSL}(2,p) \times\langle \phi\rangle$ is a proper PF-subgroup of $G$, unless $p\in\{2,3\}$ and $\phi$ is an element of order $2$ or $3$, respectively. If $p=3$, then $H$ contains a group isomorphic to $C_2\times C_2 \times C_3$, which is one of the PF-groups mentioned in the previous section. Hence, $p=2$ and $\phi$ has order $2$. Note that ${\rm PSL}(2,p^{\alpha})$ contains a subgroup $L$ isomorphic to the Frobenius group $C_{2^{\alpha}}:C_{2^{\alpha}-1}$ and $\langle \phi\rangle$  acts on $L$ by conjugation. Therefore, $T=L\langle \phi\rangle$ is a  proper  non-cyclic subgroup of $G$, and so it admits a non-trivial partition.  Looking at the structure of solvable groups admitting a non-trivial partition, we see that the only possibility for $T$ is being  either  a  group of Hughes-Thompson Type  or isomorphic to $S_4$ (it is not a Frobenius group). If the latter case occurs, then $\alpha=2$ and $G\cong {\rm PGL}(2,5)$ which is a group admitting a non-trivial partition, a contradiction. Thus, $T$ is a group of Hughes-Thompson Type, and $H_s(T)$ is a proper subgroup of $T$ of index $s$ for some prime $s$ dividing $|T|$. Regarding that $H_s(T)$ is a nilpotent group, we  conclude  that $H_s(T)$ is a Sylow $2$-subgroup of $T$ and $s=2^{\alpha}-1$. Therefore, $T\cong P: R$, where $P$ is the Sylow $2$-subgroup of $T$ and $R$  is a cyclic  subgroup of order $s=2^{\alpha}-1$. Observing Remark \ref{smallresults}(ii), we get that   $T$ must be a Frobenius group, which is not possible by the structure of $T$. Hence, $G/S$ does not contain a  non-trivial field automorphism.

	Thus, $G/S$ is generated by a field-diagonal automorphism (since ${\rm PGL}(2,p^\alpha)$ admits a non-trivial partition). Let $\phi$ be a field-diagonal automorphism generating   $G/S$.   We  may assume that $\phi$ is a $2$-element. Therefore, $p$ is odd, $G/S$ is a cyclic group of order $2$, and $\alpha$ is even.
	
	Now, let $L=\{d_{\omega} \ | \ \omega\in F^*_{p^{\alpha}}\}$, where
	
	\[ d_{\omega}= \left [\begin{array}{cc} \omega & 0 \\ 0 & \omega ^{-1} \end{array} \right]  {\bf Z}(\text{SL}(2,p^{\alpha})).\]
	
	Hence, $d_{\omega}^\phi=d_{\omega^{p^{\alpha/2}}}$ and $|L|=(p^\alpha-1)/2$. Thus, $T=L\langle\phi\rangle$ is a proper subgroup of $G$. By the assumption, $T$ is cyclic or admits a non-trivial partition. In the former case, $L\leq \textbf{C}_G(\phi)$, and so $d_{\omega}^\phi=d_{\omega}=d_{\omega^{p^{\alpha/2}}}$, for every $\omega \in F^*_{p^{\alpha}} $, implying that $o(\omega)\mid 2(p^{\alpha/2}-1)$, for all $\omega \in F^*_{p^{\alpha}}$, a contradiction. Therefore, $T$ must admit a non-trivial partition. Considering that $L$ is a cyclic group of order  divisible by $2$ ($\alpha$ is even and so $4$ divides $p^\alpha-1$) and $\phi$ is a $2$-element,  $T$ is not a Frobenius group, and so $T$ is either a $2$-group or a group of Hughes-Thompson type (by  Lemma \ref{classification}). Let $T$ be a $2$-group. Since $\alpha$ is even, $\alpha=2$ and $p=3$ by \cite[Lemma 2.10]{dr}, and so $G\cong M_{10}$. But Sylow $2$-subgroups of $M_{10}$ are isomorphic to the semidihedral group of order 16 which is impossible since it does not admit a non-trivial partition by Corollary \ref{smallresultse}(ii). Hence, $T$ is a group of Hughes-Thompson type. Since $|T:H_r(T)|\neq r$, for every odd prime divisor $r$ of $|T|$, we must have $|T: H_2(T)|=2$, and so $o(\phi)=2$ and $H_2(T)=L$. In addition, by Remark \ref{smallresults}(i), $\textbf{C}_T(\langle\phi\rangle)$ does not contain any element of order $t\not=2$. Let $\omega_0\in F_{p^\alpha}^*$  such that  $o(\omega_0)=2(p^{\alpha/2}-1)$,  then $d_{\omega_0}^\phi=d_{{\omega_0}^{p^{\alpha/2}}}=d_{\omega_0}$. Hence, $o(d_{\omega_0})=2$. Since $o(d_{\omega_0})$ is divided by $o(\omega_0)/2$, we get that $p^{\alpha/2}-1=2$, and so  $p^\alpha=9$. Therefore, $G\cong M_{10}$, causing a contradiction, by  the same  argument as before.    
\end{proof}

\begin{lemma}\label{perfect}
	Let $G$ be a non-solvable minimal PF-group. Then, $G$ is perfect.
\end{lemma}
\begin{proof}
	
	On the contrary, suppose that $G$ is a non-solvable  minimal PF-group such that  $G'<G$. Since $G'$ is a non-solvable group admitting a non-trivial partition, it must be isomorphic to the almost simple groups of types ${\rm PGL(2,p^{\alpha})}$, ${\rm PSL}(2,p^{\alpha})$  or ${\rm Sz}(2^{\alpha})$, where $p$ is a prime and $\alpha$ a natural number. Therefore, $G''$ is isomorphic to ${\rm PSL}(2,p^{\alpha})$ or ${\rm Sz}(2^{\alpha})$.

	If ${\textbf{C}}_G(G'')=1$, then $G$ is an almost simple group with socle $G''$, and so applying Lemma \ref{almostsimple}, we get a contradiction.  
	Thus, $G''<H\cong{\textbf{C}}_G(G''){\times}G''\leq G$. Looking at the structure of non-solvable groups admitting non-trivial partitions, we see that  $H$ is a PF-group.      
	Hence, we must have $G\cong{\textbf{C}}_G(G''){\times}G''$.
	Now, assume that $r$ is a prime divisor of $|{\textbf{C}}_G(G'')|$.  Let  $M\cong C_r$ be  a  subgroup of ${\textbf{C}}_G(G'')$. Let also $N$ be a subgroup of $G''$. Assume that  $N\cong D_{2(p^{\alpha}-1)/(2,p^{\alpha}-1)}$ if $G''\cong {\rm PSL}(2,p^{\alpha})$, and $N\cong {\rm Sz}(2)\cong C_5: C_4$ if $G''\cong{\rm Sz}(2^{\alpha})$. If $G''\cong {\rm Sz}(2^{\alpha})$ or $r\not=2$, then  $M\times N$ is a non-cyclic proper PF-subgroup of  $G$, by Lemma \ref{classification}, Remark \ref{smallresults}(ii) and Remark \ref{smallresults}(iv), thus  we get a contradiction. Hence, $r=2$ and $G''\cong {\rm PSL(2,p^{\alpha})}$.  Let $T$ be a Frobenius subgroup of $G''$ isomorphic to $C_{p^{\alpha}}: C_{(p^{\alpha}-1)/(2, p^{\alpha}-1)}$. Hence, $M\times T$ is a PF-group by Lemma \ref{classification}, Remark \ref{smallresults}(ii) and Remark \ref{smallresults}(iv), except for $(p^{\alpha}-1)/2=2$. Therefore, $G''\cong A_5$, and so by Remark \ref{smallresults}(ii) and Remark \ref{smallresults}(iv), $C_2\times A_4$ is a proper PF-subgroup of $G$,  a contradiction.
\end{proof}
\begin{proposition}\label{solvable}
	If $G$ is a minimal PF-group, then $G$ is solvable.
\end{proposition}
\begin{proof}
	Assume that $G$ is non-solvable and $R$ is the solvable radical of $G$.
	Let $M/R$ be a chief-factor of $G$. Therefore, $M/R\cong{S^k}$, where $S$ is a non-abelian simple group and $k\geqslant 1$ is an integer. Firstly, suppose that $M<G$. Since $S$ is non-solvable, by Lemma \ref{classification}, $M$ is isomorphic to ${\rm PSL}(2,p^h)$ or ${\rm Sz}(2^h)$, where $h\geqslant 1$ is an integer. If  ${\textbf{C}}_G(M)=1$, then $G$ is an almost simple group, which is a contradiction by Lemma \ref{almostsimple}. 
	Hence, $\textbf{C}_G(M)\neq 1$, and so ${\textbf{C}}_G(M){\times}M$ does not admit a non-trivial partition, implying that $G\cong M\times {\textbf{C}}_G(M)$. Since $M\times C_r$, where $r$ is a prime divisor of $|\textbf{C}_G(M)|$, embeds in $G$ and it is a PF-group, we conclude that $G\cong M\times C_r$, which is a contradiction since $G$ is perfect, by Lemma \ref{perfect}. 
	Therefore, $M=G$, and so $G/R\cong{S^k}$. Since $G/R$ is a chief-factor of $G$, we deduce that $k=1$.
	
	Now, we show that $G/R$ is a minimal simple group. 
	Assume that $G/R$ is not a minimal simple group, and $T/R$ is a non-solvable proper subgroup of $G/R$. Using Lemma \ref{classification}, we conclude that $T$ is isomorphic to the almost simple groups of types   ${\rm PGL}(2,p^\alpha)$, ${\rm PSL}(2,p^{\alpha})$ or ${\rm Sz}(2^{\alpha})$, where $\alpha\geqslant 1$ is an integer, so $R=1$ and $G\cong{S}$.
	
	We know that every alternating group of degree $n$ contains an alternating group of degree $n-1$ which is not isomorphic to the almost simple groups of types   ${\rm PGL}(2,p^\alpha)$, ${\rm PSL}(2,p^{\alpha})$ or ${\rm Sz}(2^{\alpha})$, where $\alpha\geqslant 1$ is an integer. Thus, $G\ncong A_n$, for $n>7$. Moreover, $A_7$ contains a PF-subgroup isomorphic to $C_2\times C_2\times C_3$ (see Corollary \ref{smallresultse}(i)), and so the only possibility  for $G$ among alternating groups is being isomorphic to $A_5$ or $A_6$, which is a contradiction (since $G$ does  not admit a non-trivial partition). By looking  at the maximal subgroups of sporadic groups and Tits group \cite{atlas}, we get that $G$ is not a sporadic or Tits group. Furthermore, \cite[Lemma 6]{reegroups} and \cite[Main Theorem]{reegroups2} imply that $G$ is not a Ree group. Let $G\cong {\rm PSL}(n,q)$, for some prime power $q$. Note that by Lemma \ref{classification}, ${\rm PSL}(2,q)$ admits a non-trivial partition. Notice also that ${\rm PSL}(3,3)$ is a minimal simple group. Hence, we do not consider them.  By the fact that SL$(n,q)$ contains a subgroup isomorphic to SL$(n-1,q)$ whose intersection  with the center of SL$(n,q)$ is trivial, we deduce that   ${\rm PSL}(n,q)$ contains a subgroup isomorphic to ${\rm SL}(n-1,q)$. Hence, by our assumption, we get a contradiction in case $n\neq3$ or $q$ is not a power of 2. Therefore, the only possibility for $G$ among projective special linear groups is $G\cong {\rm PSL}(3,q)$, where $q$ is a power of 2. For $q>4$, by \cite[Theorem 19]{psl3}, $G$ contains a subgroup isomorphic to $[q^2]:{\rm GL}_2(q)$, a contradiction. In addition, since ${\rm PSL}(3,2)\cong {\rm PSL}(2,7)$ is not a PF-group, and ${\rm PSL}(3,4)$ contains a non-solvable subgroup isomorphic to $2^4:A_5$, we get a contradiction. Hence, $G$ is not isomorphic to any projective special linear group. 
	Moreover, by \cite[Table 0A8]{table}, $G$ is not isomorphic to any twisted Chevalley group or Chevalley group, besides $B_2(q)$ and ${\rm PSU}(3,q)$, for some prime power $q$. Firstly, suppose that $G\cong B_2(q)={\rm PSp}(4,q)$. For $q$ odd, by \cite{psp}, $G$ contains a subgroup which is the central product of two groups isomorphic to ${\rm SL}(2,q_1)$ and ${\rm SL}(2,q_2)$, for suitable $q_1$ and $q_2$, a contradiction. If $q$ is even, then ${\rm PSp}(4,2)\cong S_6$ embeds in ${\rm PSp}(4,q)$, a contradiction.
	Finally, assume that $G\cong {\rm PSU}(3,q)$. If $q$ is odd, then by the structure of Sylow 2-subgroups of $G$ (see \cite[p.2]{psu3odd}), and Corollary \ref{smallresultse}(ii), we get a contradiction. Also, for $q$ even, by \cite[Lemma 2.3]{psuqeven}, $G$ contains a subgroup isomorphic to ${\rm PSL}(2,q)\times (q+1)/(3,q+1)$, a contradiction.
	
	Consequently, $G/R\cong{S}$ is a minimal simple group. Using Thompson's classification of minimal simple groups, one of the following cases occurs:
	
	$\bullet$ Let $G/R$ be isomorphic to ${\rm PSL}(2,q)$, where $q=p$, $2^p$ or $3^p$, for some prime $p$. We know that $G/R$ contains subgroups $T_1/R\cong{D}_{2d}$ and $T_2/R\cong{D}_{2f}$, where $d$ and $f$ are distinct divisors of $q-1$ and $q+1$, respectively. Therefore, $T_i$ must be either a Frobenius group, a group of Hughes-Thompson type, a 2-group or isomorphic to $S_4$, for $1\leqslant i\leqslant2$. 
	
	If $T_i$ is isomorphic to $S_4$, for some $1\leqslant i\leqslant 2$, then $R\cong C_2\times C_2$, and so  applying  Normalizer-Centralizer Theorem on $R$, we get that  $R=\textbf{Z}(G)$. Note that by Lemma \ref{perfect}, $G$ is perfect, and so $R\leq M(S)$, where $M(S)$ is the Schur multiplier of $S$, a contradiction since $|M(S)|=1,2,3$ or 6. Now, let $C_d\cong M_1/R\lhd T_1/R$ and $C_f\cong M_2/R\lhd T_2/R$. 
	If for some $i\in\{1,2\}$, $T_i$ is a Frobenius group, then $M_i$ is a subgroup of the Frobenius kernel. Also, if $T_i$ is a group of Hughes-Thompson type, then regarding Remark \ref{smallresults}(iii)   $M_i\leq H_2(T_i)$, and so $M_i$ is a nilpotent group. 
	
	Consequently, in all remaining cases, $M_i$ is nilpotent.
	Now, let $\overline{G}=G/R'$ and $\overline{R}=R/R'$. Thus, $\overline{R}={\overline{R}_d}\times{\overline{R}_{d'}}$, where ${\overline{R}_{d}}$  and ${\overline{R}_{d'}}$ are   Sylow $d$-subgroup and  $d$-complement of $\overline{R}$, respectively. Since $\overline{M_1}\leq{{\textbf{C}}_{\overline{G}}({\overline{R}_{d'}})}$, we get that ${{\textbf{C}}_{\overline{G}}({\overline{R}_{d'}})}=\overline{G}$, and so ${\overline{R}_{d'}}\leq{\textbf{Z}}(\overline{G})$. Similarly, ${\overline{R}_{f'}}\leq{\textbf{Z}}(\overline{G})$, and so $\overline{R}={\textbf{Z}}(\overline{G})$. Hence,  $\overline{G}/\overline{R}\cong{S}$ implies that  $\overline{R}\leq M(S)$ in which $ M(S)$ is trivial or isomorphic to $C_2, C_3$ or $C_6$.
	
	If $\overline{R}=1$, then $G\cong {\rm PSL}(2, q)$, a contradiction since PSL$(2,q)$ admits a non-trivial partition. If $\overline{R}\cong C_2$, then $R$ is a 2-group, and so $R=\overline{R}$. Hence, $G\cong {\rm SL}(2,q)$. Obviously, $q$ is odd, and so $G$ contains a subgroup isomorphic to the generalized quaternion group,  which is a contradiction by Corollary \ref{smallresultse}(ii). Hence, $\overline{R}\cong C_3$ or $C_6$, and $q=9$. 
	Thus, $G\cong 3. A_6$ or $6. A_6$. In both cases, $G$ contains a group isomorphic to  $H=C_3\times P$, where $P\in $ Syl$_2(G)$. But, regarding Lemma \ref{classification} and Remark \ref{smallresults}(iv), we get that $H$ does not admit a non-trivial partition, a contradiction.

	$\bullet$ Let  $G/R\cong S={\rm Sz}(q)$ with $q=2^p$, $T_1/R\cong{D}_{2(q-1)}$ and $T_2/R\cong{C}_{q+2r+1}:{C_4}$, where $p=2k+1$ and $r=2^k$. Therefore, similar to the above discussion, $R$ is nilpotent and  $\overline{R}=R/R'\leq M(S)$. If $p>3$, then $\overline{R}=1$, and so $G\cong \rm{Sz}(2^p)$ contradicting that $G$ is a PF-group. Thus, $G/R\cong{{\rm Sz}(8)}$ and $|\overline{R}|\not =1$ is a divisor of $4$. Let $L/R$ be a subgroup of $G/R$ isomorphic to the Frobenius group $C_5:{C_4}$. Since $R$ is a 2-group, $L$ is not a Frobenius group. Also  it cannot be  a group of Hughes-Thompson type by Remark \ref{smallresults}(iii). Hence, it does not admit a non-trivial partition, a contradiction.

	$\bullet$ Let $G/R\cong S= {\rm PSL}(3,3)$ and $T/R\cong{S}_{4}$. Note that $T$ must admit a non-trivial partition.  Regarding  that a Frobenius group does not have a factor isomorphic to $S_4$ and  observing Remark \ref{smallresults}(iii) imply that  the only possibility for $T$ is being isomorphic to $S_4$. Hence, $R=1$, and so  $G \cong{{\rm PSL}(3,3)}$.  Therefore, $G$ contains a maximal subgroup isomorphic to $3^2:2S_4$ which is a PF-group; in fact it is neither a Frobenius group nor isomorphic to $S_4$, also it does not have any normal nilpotent subgroup with prime index which means it is not a group of Hughes-Thompson type, a contradiction.
\end{proof}
\section{\bf{Classification of minimal PF-groups}}
In order to get the main result, we use the following proposition several times. 
\begin{proposition}\label{p3}
	Let $G$ be a finite $p$-group whose all proper non-cyclic subgroups admit non-trivial partitions. Let  $G$ have a non-cyclic maximal subgroup  and an element of order $p^a$,  for some $a>1$, outside of the  non-cyclic maximal subgroup. Then, $G$ is isomorphic to $C_p\times C_{p^2}$, $M_3(p)$ or $D_{2^n}$, for some natural number $n\geqslant 3$. 
\end{proposition}
\begin{proof}
	Let $M$ be a non-cyclic maximal subgroup of $G$ and $y\in G\setminus M$ be an element of order $p^a$, for some $a>1$. 
	If $|M|=p^2$, then $|G|= p^3$, and so by the structure of groups having order $p^3$, we get the result. Hence, in the sequel, we suppose that $|M|\geq p^3$, and so $|G|\geqslant p^4$. Since $M$ admits a non-trivial partition, $H_p(M)<M$. We claim that $H_p(M)$ is not trivial.

	On the contrary, suppose that $H_p(M)=1$. Obviously, $o(y)=p^2$. Now, let $L_1={\bf Z}(G)\langle y\rangle$. Since $L_1$ is abelian and contains an element of order $p^2$, by the assumption that proper  non-cyclic subgroups of $G$ admit non-trivial partitions and Corollary \ref{smallresultse}(i), $L_1$ must be cyclic or $L_1=G$. If the latter case occurs, then,  as    $G$ is not cyclic,  $G$ has a subgroup  isomorphic to $ C_p\times C_{p^2}$. Thus, Corollary \ref{smallresultse}(i) yields that  $G\cong C_p\times C_{p^2}$, a contradiction. Hence, $L_1$ is cyclic. Moreover,  $L_1=\langle y\rangle$ since otherwise $y\in \textbf{Z}(G)$, and so by Corollary \ref{smallresultse}(i), $\langle x\rangle\times \langle y\rangle$, where $x\in M\setminus\langle y\rangle$,  is a proper PF-subgroup of $G$. Hence, $\textbf{Z}(G)=\langle  y^p \rangle $ since $y\notin \textbf{Z}(G)$.
	Let $N_1$ be a normal subgroup of $G$ of order $p^3$ contained in $M$. Then, $L_2=N_1L_1$ is  a subgroup of $G$.
	Note that if $L_2$ is abelian, then by the assumption and Corollary \ref{smallresultse}(i), $L_2$ must be cyclic or $L_2=G$. Note that  $N_1\leq L_2$ is a group of order $p^3$ of exponent $p$, and so $L_2$ cannot be cyclic. In addition, if $L_2=G$, then it causes a contradiction with the same argument as the one we have above. Hence, $L_2$ is  non-abelian.  Now, consider $L_2/{\bf Z}(G)$ which is a group of order $p^3$ and exponent $p$. Thus, $L_2/{\bf Z}(G)\cong N_1/{\bf Z}(G)\rtimes L_1/{\bf Z}(G)$. We claim that $L_2/{\bf Z}(G)$ is abelian. On the contrary, suppose that $L_2/{\bf Z}(G)$ is non-abelian. Let $K:={\bf{Z}}_2(L_2)L_1\leq L_2$, where ${\bf{Z}}_2(L_2)$ is the second center of $L_2$. Note that $|K:L_1|=p$ and $L_1$ is cyclic. Hence, by Lemma \ref{p-groups}, Corollary \ref{smallresultse}(ii), and the assumption that non-cyclic proper subgroups of $G$ admit non-trivial partitions, we conclude that $K\cong D_8$. Now, consider ${\bf{C}}_{L_2}(K)$. Obviously, ${\bf Z}(K)={\bf{C}}_{L_2}(K)\cap K \cong {\bf{Z}}(D_8)$, and so ${\bf Z}(K)={\bf{Z}}(G)$. In addition, if ${\bf Z}(K)<{\bf C}_{L_2}(K)$, then ${\bf C}_{L_2}(K)$ contains a subgroup $T$ of order $p$ outside of $K$, and so  $T{\bf Z}_2(L_2)/{\bf Z}(G)
	<{\bf{C}}_{L_2/{\bf Z}(G)}(L_1/{\bf Z}(G))$  is a group of order $p^2$ whose intersection with $L_1/{\bf Z}(G)$ is trivial. Hence,  $L_1/\textbf{Z}(G)$ acts trivially on $N_1/\textbf{Z}(G)$, a contradiction. Thus, $L_2/{\bf Z}(G)$ is an elementary abelian group, and ${\bf Z}(G)=L_2'$. Now, consider the abelian group ${\bf Z}(L_2)L_1$. Similar to the above discussion, by Corollary \ref{smallresultse}(i), we get that ${\bf Z}(L_2)L_1$ must be cyclic. Moreover, we may assume that ${\bf Z}(L_2)L_1=L_1$ since otherwise $y\in \textbf{Z}(L_2)$, and so $\langle x\rangle\times \langle y\rangle$, where $x\in N_1\setminus\langle y^p\rangle$, is a proper PF-subgroup of $G$, by Corollary \ref{smallresultse}(i). Hence, ${\bf Z}(L_2)={\bf Z}(G)=L_2'$, and so $L_2$ is an extraspecial group of order $p^4$, a contradiction. 
	
	Therefore, $H_p(M)>1$. Since  $H_p(M)$ admits a non-trivial partition, and $H_p(H_p(M))$ $=H_p(M)$, we get that $H_p(M)$ is cyclic.    On the other hand, $H_p(H_p(M)\langle y \rangle)$=$H_p(M)\langle y \rangle$, implying that $H_p(M)\langle y \rangle$ does not admit a non-trivial partition. Thus,  $H_p(M)\langle y \rangle=G$ or $H_p(M)\langle y \rangle$ is cyclic.   
	
	First, assume that $G=H_p(M)\langle y \rangle$. We show that  $ y^p\in H_p(M)$.  As,  $({\bf Z}(G)\cap H_p(M))\langle y \rangle$  is abelian, we conclude that $({\bf Z}(G)\cap H_p(M))\langle y \rangle= \langle y \rangle $, by the same argument we have in the above.   If $o(y)=p^2$, then $y^p\in ({\bf Z}(G)\cap H_p(M))$, as wanted. If $o(y)>p^2$, then as $y^p\in M$, we get that $y^p\in H_p(M)$, as desired.
	Hence, $H_p(M)$  is a maximal  subgroup of $G$, contradicting $H_p(M)\not =M$. 
	
	Thus, $H_p(M)\langle y \rangle$ is cyclic. On the other hand,  $M$ contains a subgroup $T$ such that $H_p(M)< T$ and $|T:H_p(M)|=p$. By our hypothesis, $T$ admits a non-trivial partition and by Lemma \ref{p-groups} and Corollary \ref{smallresultse}(ii), we get that $T\cong D_{2^m}$, for some natural number $m\geqslant 3$.
	
	Consequently, $p=2$, $G$ is a 2-group, and by \cite[p.1]{2groups}, $|G:H_2(M)|=4$. Thus, $H_2(M)\langle y\rangle$ is a  cyclic maximal subgroup of $G$. Lemma \ref{p-groups} implies that  $G$ is isomorphic to $D_{2^n}$, for some natural number $n\geqslant 3$.
\end{proof}

{\bf Proof of the main theorem.}
According to Section 3, we just need to prove the  "only if" part of the main theorem. Note that by Proposition \ref{solvable}, $G$ is solvable. Assume that $M$ is a maximal normal subgroup of $G$, and so $|G:M|=p$, where $p$ is a prime dividing $|G|$. Since $G$ is a minimal PF-group, by Lemma \ref{classification}, one of the  following  cases occurs:

\bigskip

\textbf{Case (I)\,} Let $M$ be cyclic. 

\bigskip

Firstly, suppose that $|G|=p^n$. Then, $G$ satisfies the hypotheses of  Lemma \ref{p-groups}. Regarding Corollary \ref{smallresultse}(ii), $D_{2^n}$ admits a non-trivial partition for every $n\geqslant 3$, and so $G\ncong{D_{2^n}}$. In addition, by the fact that the generalized quaternion groups contain the quaternion group of order $8$, and $Q_8$ does not admit a non-trivial partition, the only possibility for $G$ among the generalized quaternion groups is being isomorphic to ${Q_8}$.  Similarly, $G$ is not isomorphic to the semidiherdral group of order $2^n$. Therefore, it remains to consider the groups $C_p\times{C_{p^{n-1}}}$ and $M_n(p)$ for $n\geqslant 3$. Note that by Corollary \ref{smallresultse}(i), $C_p\times{C_{p^{\alpha}}}$, with $\alpha\geqslant 2$, does not admit a non-trivial partition. Also, $M_n(p)$, with $n\geqslant 4$, contains a subgroup isomorphic to $C_p\times{C_{p^{n-2}}}$. Thus, the only possibility for $G$ among remaining groups is being isomorphic to $M_3(p)\,\, (p\neq 2)$ or $C_p\times{C_{p^2}}$. Hence, in this case, $G\cong Q_8$, $M_3(p)\,\, (p\neq 2)$ or $C_p\times{C_{p^2}}$,  the  groups described in part (1)-(3) of the main theorem.

Now, assume that $G$ is not a $p$-group, and $|G|_p=p^n$, for some natural number $n$. By the assumption, $G\cong F\rtimes{P}$, where $P\in{{\rm Syl}_p(G)}$ and $F$ is a cyclic $p$-complement in $G$. We know that $G\cong (F_1: P)\times F_2$, where $F_1=[P, F]$ and $F_2=C_F(P)$ (by Fitting's Lemma \cite[Theorem 4.34]{isaacs}). First, let $n=1$. Since $G$ is a minimal PF-group, $G$ is not a cyclic group or a Frobenius group, and so $|F_1|,|F_2|>1$. In addition, if $q$ and $r$ are  prime divisors (not necessarily distinct) of $|F_1|$ and $|F_2|$, respectively, then $(C_q: C_p)\times C_r$ is  the described group in example (5) or (7) of Section 3, is a PF-subgroup of $G$. Hence, $G\cong (C_q: C_p)\times C_r$ satisfies part (5) or (7) of the main Theorem. Now, assume that $n>1$. By Lemma \ref{p-groups} and the assumption, $P$ is isomorphic to $C_{p^n}$, $C_p\times{C_p}$ or $D_{2^n}$. 

Let  $P$ be  cyclic. We claim that $G$ is a minimal non-cyclic group. By the assumption, $G$ is not cyclic, and so there exists a prime $q$ dividing $|F|$ such that $P$ acts on the Sylow $q$-subgroup of $F$ non-trivially. Using Fitting's Lemma (\cite[Theorem 4.34]{isaacs}), the action of $P$ on the subgroup of $F$ of order $q$, say $Q$, is also non-trivial. Since $n>1$ and $M$ is cyclic, $H\cong Q: P$ is not a Frobenius group.  Therefore, regarding the structure of solvable groups admitting a non-trivial partition and Remark  \ref{smallresults}(ii), $H$ is a PF-subgroup of $G$, implying that $G\cong Q: P$. Thus, it is easy to see that $G$ is a quasi-Frobenius group in which the Frobenius kernel of $G/{\bf Z}(G)$ is a group of order $q$, and Frobenius complements of $G/{\bf Z}(G)$ are isomorphic to $C_p$. Moreover, ${\bf Z}(G)\cong C_{p^{n-1}}$. Now, we show that all proper subgroups of $G$ are cyclic. Note that $p$-subgroups and  the $q$-subgroup of $G$ are cyclic. Thus, consider a proper subgroup of $G$ whose order is divided by $pq$, say $T$. If a Sylow $p$-subgroup of $T$ contained in  ${\bf{Z}}(G)$, then $T$ is cyclic. Otherwise, $|P_T|=|P|$, where $P_T\in {\rm Syl}_p(T)$, implying that $T=G$, a contradiction. It means that all proper subgroups of $G$ are cyclic. Therefore, the described group $G$, is a minimal non-cyclic group. By \cite{minimalcyclic}, $G$ is isomorphic to the described group in part (6) of the main theorem.

Let $P\cong C_p^2$.  Hence,  according to Fitting's  Lemma,  either $G$ has a subgroup isomorphic to  $H={C_{q}}\times{C_p^2}$,  for some prime $q$ dividing $|F|$, or $G\cong T\times{C_{p}}$, where $T$ is a Frobenius group isomorphic to ${C_{|F|}}:{C_p}$. The latter case  does not occur,  since  by Remark \ref{smallresults}(ii) and Remark \ref{smallresults}(iv),  $G$ is a group of Hughes-Thompson type, and so it admits a non-trivial partition. Hence, the former case holds.  Since by Corollary \ref{smallresultse}(i), $H$ does not admit a non-trivial partition, $G=H$, which is the group  described in part (4) of the main theorem.

Let $P\cong D_{2^n}$. Then, $G\cong C_{|F|}\rtimes{D_{2^n}}$. Thus,  we have $G\cong D_{2^nf'}\times C_{{|F|}/f'}$ in which $f'$ is a divisor of ${|F|}$.  If $f'\neq |F|$, then $H\cong C_2^2\times C_{|F|/f'}$, by Corollary \ref{smallresultse}(i), is a PF-subgroup of $G$, a contradiction. Consequently, $G\cong D_{2^n{|F|}}$ which is a group of Hughes-Thompson type ($H_2(G)$ is a cyclic group of index $2$ in $G$), a contradiction.

\bigskip

\textbf{Case (II)\,} Let $M$ be a non-cyclic $q$-group with $H_q(M)\neq{M}$, for some prime $q$. 

\bigskip

Firstly, suppose that $q=p$, and so $G$ is a $p$-group. Since $H_p(M)<M$ and $H_p(G)=G$, we get that $G$ satisfies the hypotheses of Proposition \ref{p3}. Hence, $G\cong M_3(p)$ with $p\geqslant 3$ or $C_p\times C_{p^2}$, as desired.

Now, assume that $q\not=p$, and so $|G|=q^\alpha p$, where $\alpha$ is a natural number. Then, $G\cong{M}\rtimes{P}$, where $P$ is a Sylow $p$-subgroup of $G$.
Let $C={\textbf{C}}_{M}(P)$ and $L=CP\cong C\times P$. Note that $C$ is not trivial (since otherwise $G$ is a Frobenius group and so $G$ admits a non-trivial partition).  Note that   $C$ can  not be  isomorphic to any generalized quaternion group, since otherwise it does not admit a non-trivial partition. If $C$ is not cyclic, then by Corollary \ref{smallresultse}(i), $G$ contains a PF-subgroup isomorphic to $C_q^2\times C_p$, implying that $G\cong C_q^2\times C_p$, as wanted. Hence, we may assume that $C$ is cyclic, and so $C<M$. 

Now, suppose that $G$ contains a minimal normal subgroup $N$ such that $L\cap N=1$. Since $NL$ is not a Frobenius group, a cyclic group or isomorphic to $S_4$, we deduce that either $NL$ is a group of Hughes-Thompson type or $G=NL$.  Taking into account   Remark \ref{smallresults}(ii), the former case is not possible. 
Thus, $G=NL\cong C_q^{a}: C_{pq^b}$, for some natural numbers $a$ and $b$.
Note also that by the minimality of $N$, $N\leq {\bf Z}(M)$.  Thus, $G\cong (C_q^a : C_p)\times C_{q^{b}}$. If $b>1$, then $G$ contains a PF-subgroup which is a contradiction.  Consequently,  $G\cong (C_q^a : C_p)\times C_{q}$, which is the described group in part (7) of the main theorem.  

In the sequel, assume that all minimal normal subgroups of $G$ have a non-trivial intersection with $L$. We claim that $H_q(M)=1$.
On the contrary, suppose that $H_q(M)\neq 1$, and so $H_q(M)\cap C\neq 1$. Noting that $H_q(H_q(M))=H_q(M)$, we get that $H_q(M)$ is cyclic, and so $H_q(M)P$ must be a proper cyclic subgroup of $G$ (by applying Fitting's Lemma). Consequently, $H_q(M)\leq C < M$, and so $H_q(M)=C$ since $C$ is cyclic. Let $T/H_q(M)$ be a minimal normal subgroup of $G/H_q(M)$. Hence, $T/H_q(M)$ is either a cyclic group of order $p$ or an elementary abelian $q$-group. If ${T/H_q(M)}\cong{C_p}$, then $T$ is cyclic, and  so $G\cong M\times P$, a contradiction since $C=H_q(M)< M$. Thus, ${T/H_q(M)}$ is an elementary abelian $q$-group. Suppose that $PT<G$. Since $PT$ is not a cyclic group, a Frobenius group or isomorphic to $S_4$, by the assumption, the only possibility for $PT<G$ is being a group of Hughes-Thompson type. Now, looking at Remark \ref{smallresults}(ii) we get a contradiction. 
Consequently, $G=PT$. Note that $\textbf{Z}(G)$ is a $q$-group since otherwise $P\unlhd G$, and so $M=C$, a contradiction. Hence, ${\bf Z}(G)\leq C=H_q(M)$. Note that $\textbf{C}_G(H_q(M))\cong \textbf{C}_M(H_q(M))\rtimes P$. If $H_q(M)<\textbf{C}_M(H_q(M))<M$, then exactly similar to the above discussion (we had for $PT$), we get that $\textbf{C}_G(H_q(M))$ is a proper PF-subgroup of $G$, a contradiction. Thus, $\textbf{C}_M(H_q(M))= H_q(M)$ or $ M$. If $\textbf{C}_M(H_q(M))= H_q(M)$, then $P$ is a normal subgroup of $ C_G(H_q(M))$,  implying that $P$ is normal in $G$ and so $M=C$ which is a contradiction. Hence, $C_M(H_q(M))= M$, implying that $H_q(M)=\textbf{Z}(G)$. As a result, $G$ contains a proper PF-subgroup isomorphic to $H_q(M)\times C_q$, a contradiction. Thus, $H_q(M)=1$, and so the Sylow $q$-subgroup of $G$ has exponent $q$.

Let $H\leq \textbf{Z}(M)\cap  C$ be a group of order $q$, and $S/H$ be a minimal normal subgroup of $G/H$. If $S/H$ is a group of order $p$, then $P\unlhd G$, a contradiction.  Hence, $\overline{S}=S/H$ is an elementary abelian $q$-group. Let also $\overline{P}=PH / H$ and $K/H=\textbf{C}_{\overline{S}}(\overline{P})$. If $K/H\neq 1$, then since $C$ is cyclic and $H_q(M)=1$, we must have $|K|=q$, a contradiction. 
Consequently, $PS/H$ is a Frobenius group. Hence, looking at the structure of solvable groups admitting a non-trivial partition and Remark \ref{smallresults}(ii), $PS$ is a PF-group, and so we conclude that $G=PS$ and $H={\bf Z}(G)$. Therefore, $G$ satisfies part (7) of the main theorem. 

\bigskip

\textbf{Case (III)\,} Suppose that $M$ is a Frobenius group in which $K$ is the Frobenius kernel and $H$ is a Frobenius complement.

\bigskip

We pursue this case, by proving the following steps:

\bigskip

{\bf Step 1.} $H$ is cyclic and $K$ is either a cyclic group or a $q$-group for some prime $q$. 

\bigskip

First, we show that $H$ is cyclic.  It is well-known that Sylow subgroups of Frobenius complements are either cyclic groups or     generalized quaternion groups.  Since generalized quaternion groups are PF-groups, the Sylow subgroups of $H$ must be cyclic. Thus, we may assume that $|H|$ has at least two distinct prime divisors. It also is well-known that Frobenius complements and their subgroups cannot be Frobenius, and so $H$ is not a Frobenius group or isomorphic to $S_4$. Therefore, $H$ is either a group of Hughes-Thompson type or a cyclic group.  Assume that $H_r=H_r(H)$ and  $|H:H_r|=r$, for some prime divisor $r$ of $|H|$. 
Since the Sylow subgroups of $H$ are cyclic, $r\nmid |H_r|$. Consequently, $H\cong H_r: R$, where $R\cong C_r$, and so applying Remark \ref{smallresults}(ii)  $H$ must be a Frobenius group, a contradiction. As a result, $H$ is cyclic, as wanted. 

Furthermore, since $K$ is nilpotent, by the expressed argument following Lemma \ref{classification}, it is either a cyclic group or a $q$-group for some prime divisor $q$ of $|K|$.

\bigskip

{\bf Step 2.} If $p$ is a divisor of $|M|$, then $p$ divides $|K|$.

\bigskip

On the contrary, suppose that $p$ is a divisor of $|H|$, and $H_0$ is a complement of $K$ in $G$ containing $H$. Hence, $|P_{H_0}:P_H|=p$, where $P_{H_0}$ and $P_H$ are Sylow $p$-subgroups of $H_0$ and $H$, respectively. Since $P_H$ is cyclic, by Lemma \ref{p-groups}, $P_{H_0}$ is a cyclic group, an elementary abelian group of order $p^2$ or a dihedral group of order $2^n$, for some natural number $n\geqslant 3$.

If $P_{H_0}$ is cyclic, then using  Remark \ref{smallresults}(ii), $KP_{H_0}$ must be a Frobenius group or equal to $G$ (it can not be isomorphic to $S_4$). Let $G=KP_{H_0}$ and $x$ be a generator of $P_{H_0}$. Then,  since $G$ is not a Frobenius group, $1\neq C_{K}(x)\leq C_{K}(x^p)$, contradicting that $\langle x^p\rangle=P_H$.  Thus, $KP_{H_0}$ and so  $G$ is a Frobenius group, a contradiction.

Therefore, in the remaining possibilities for $P_{H_0}$, we have $G\cong M\rtimes P_1$,  where $P_1\cong C_p$. Note that if $KP_1$ is a  group of Hughes-Thompson type, then  by Remark \ref{smallresults}(ii), $KP_1$ is a   Frobenius group which will be considered. Hence, $KP_1$ is either a Frobenius group or a cyclic group (notice that $KP_1$ is a proper subgroup of $G$ and it is not isomorphic to $S_4$). If $KP_1$ is a Frobenius group, then $G$ is a Frobenius group, a contradiction. Thus, $KP_1$ is cyclic. As, $KP_1=C_G(K)\unlhd G$, we have $P_1\unlhd G$ and so $G/K\cong P_1\times H$.
Consequently, $P_{H_0}\not \cong D_{2^n}$. Note that if there exists a prime divisor $r\neq p$ of $|H|$, then by Remark \ref{smallresults}(ii) and Remark \ref{smallresults}(iv), $G$ contains a proper PF-subgroup isomorphic to $P_1\times (K: R)$, where $|R|=r$, a contradiction. Therefore, $H$ is a $p$-group, and noticing that  $P_{H_0}=C_p^2$, we have  $G\cong C_p\times (K: C_p)$ which is a group of Hughes-Thompson type by Remark \ref{smallresults}(ii) and Remark \ref{smallresults}(iv), a contradiction.

\bigskip

{\bf Step 3.} If $p$ is a divisor of $|M|$ and $K$ is non-cyclic, then $o(x)_p=p$ for all $x\in G\setminus M$. 

\bigskip

By Step 2, $p$ is a divisor of  $|K|$. Suppose that $K$ is a non-cyclic group. Then, $K$ is a  $p$-group with $H_p(K)\neq K$. Assume also that there exists an element $x\in G\setminus M$ of order $p^a$, for some natural number $a>1$. Let $P$ be a Sylow $p$-subgroup of $G$ containing $x$. Therefore, $P$ satisfies the hypotheses of Proposition \ref{p3}, and so $P\cong D_{2^n}$, for some $n\geqslant 3$. On the other hand, $H_2(K)$ is cyclic since $H_2(H_2(K))=H_2(K)$. Therefore, $H_2(K)=1$ since otherwise $H$ acts Frobeniusly on the subgroup of order two of $H_2(K)$. As a result, $K$ is an elementary abelian $2$-group, and by the structure of $P$, the only possibility for $K$ is being isomorphic to $C_2\times C_2$. Consequently, $|H|=3$, $M\cong A_4$, and $|G/M|=2$. 
Therefore, $G\cong S_4$, a contradiction,  as $S_4$ admits a non-trivial partition.  Hence, we have ${o(x)}_p=p$, for every element $x\in G\setminus M$.

\bigskip

{\bf Step 4.} If $p$ is a divisor of $|M|$ and $K$ is non-cyclic, then $G$ satisfies part (7) of the main theorem.

\bigskip

If $G/K$ is a Frobenius group, then by Step 3, $o(x)=p$, for every $x\in G\setminus M$, and so $M\cup \{\langle x \rangle | {x\in P\setminus M} \}$ forms a non-trivial partition for $G$, a contradiction. 
Thus, there exists $x_0\in G\setminus M$ of order $p|H_0|$, such that 
$C=\textbf{C}_G({x_0}^{|H_0|})\cong P_0: H_0$, where $P_0\in {\rm Syl}_p(C)$ and $H_0=C_H(x_0^{|H_0|})$.
Note that $|P_0|>p$, and so  $|{\bf Z}(C)|=p$. Consequently, $C\cong \langle {x_0}^{|H_0|}\rangle \times (K_0: H_0)$, where $1\neq K_0\leq K$. By Remark \ref{smallresults}(ii) and Remark \ref{smallresults}(iv), $C$ is a PF-group,
implying that $C=G$. Thus,
$G\cong C_p\times M$, and so $M$ must be a minimal Frobenius group since otherwise $G$ contains a proper PF-subgroup. 
Hence, $G\cong C_p\times (C_p^a : C_q)$, which is the described group in part (7) of the main theorem.

\bigskip

{\bf Step 5.} If $p$ is a divisor of $|M|$, then   $G$ satisfies part (7) of the main theorem. 

\bigskip

By Step 4, we may assume that $K$ is cyclic and by step 2, $p$ is a divisor of $|K|$. By Lemma \ref{p-groups}, $P\in {\rm Syl}_p(G)$ is a cyclic group, an elementary abelian group of order $p^2$ or a dihedral group of order $2^n$, for some natural number $n\geqslant 3$. 

If $P\cong{D_{2^n}}$, then $H$ acts  Frobeniusly on the Sylow $2$-subgroup of $K$, and so on its subgroup of order $2$, which is a contradiction.

In the two other cases, $\textbf{C}_G(K_p)=K_{p'}P$,  where  $K_p$ and $K_{p'}$ are the Sylow $p$-subgroup and  $p$-complement of $K$, respectively. Therefore,  $H$ acts co-primely on $\textbf{C}_G(K_p)$. Without loss of generality, we may assume that $P$ is $H$-invariant.  Then, $PH$ contains a subgroup $PL$, where $L\leq H$ is of prime order. 

If $P$ is cyclic, then by Fitting's Lemma (\cite[Theorem 4.34]{isaacs}), $PL$ is a cyclic group or a Frobenius group. Regarding that $P$ has a non-trivial intersection with $K$ the first case does not occur. Also,  taking into account that    $G/K$ contains a normal Sylow $p$-subgroup  $C_G(K_p)/K\cong PK/K$ and a normal $p$-complement $M/K\cong H$,  we have $PL$ cannot be a Frobenius group, a contradiction.  

Hence, $P\cong{C_p\times C_p}$.  Recalling that $HK/K$ and $PK/K$ are normal in $G/K$, we obtain that $G=PL\cong C_p\times (C_p: L)$ is satisfying  example (7) of Section 3, as wanted.

\bigskip

{\bf Step 6.}  If $p$ is not  a divisor of $|M|$, then $G$ satisfies part (5) of the main theorem. 

\bigskip

Assuming that $p$ is not a divisor of $|M|$,  $K_0\cong K\rtimes P$, where $P\in {\rm Syl}_p(G)$. Obviously, $K_0$ is not isomorphic to $S_4$. In addition, if $K_0$ is a Frobenius group, then $HP$ acts Frobeniusly on $K$, implying that $G$ is a Frobenius group, a contradiction. Therefore, $K_0$ must be a cyclic group or a group of Hughes-Thompson type.  Using Remark \ref{smallresults}(ii), the latter case does not occur, and so $K_0$ is cyclic. Note also that $K_0=C_G(K)\unlhd G$. As a result,  $G\cong M\times P$. Assume that $L\cong C_q: C_r\leq M$  is a  minimal Frobenius group, then    $G$ contains a  subgroup isomorphic to $L\times P$, the described group in example (5) of Section 3.  Hence, we must have $G\cong C_p\times ( C_q:{C_r})$, where $p,q$ and $r$ are distinct primes, as described in part (5) of the main theorem.

\bigskip

\textbf{Case (IV)\,} Suppose that $M$ is a group of Hughes-Thompson type,  with $H_q(M)<M$ for some prime $q$.

\bigskip

We proceed through this case by showing the following steps:

\bigskip

{\bf Step 1.} $H_q(M)$ is cyclic and it  is divided by $q$ and a prime other than $q$. Moreover $Q$ is isomorphic to either  $C_q^2$ or $D_{2^n}$, $n\geq 3$, where $Q\in $ Syl$_q(M)$.

\bigskip

In this case  $|M:H_q(M)|=q$ and  $M\cong H_q(M):{C_q}$. If $q$ does not divide $|H_q(M)|$, then regarding Remark \ref{smallresults}(ii), $M$ is a Frobenius group discussed in the previous case. Thus, we may assume that $q$ divides $ |H_q(M)|$. Since $H_q(M)$ is nilpotent, by the argument following Lemma \ref{classification}, $H_q(M)$ is either a $q$-group or a cyclic group. The first case implies that $M$ is a $q$-group discussed before. 
Hence, in the sequel, let $H_q(M)$ be a cyclic group whose order  is divided by  at least two distinct prime divisors. Then,  $Q\cap H_p(M)$ is a  cyclic maximal subgroup of $Q$,  for $Q\in {\rm Syl}_q(M)$. Hence, according to Corollary \ref{smallresultse}(ii),  $Q$ is isomorphic to either $C_q^2$ or $D_{2^n}$, for some natural number $n\geqslant 3$ ($Q$ is not cyclic since otherwise $Q\leq H_q(M)$). 

\bigskip

{\bf Step 2.} $p$ is a divisor of $|M|$.

\bigskip

On the contrary, suppose that $p\nmid |M|$. Let $K$ be a Hall $\{p,q\}$-subgroup of $G$, and so $Q=K\cap M\unlhd K$. Hence, $K\cong Q\rtimes P$, where $P\in {\rm Syl}_p(G)$. If $Q\cong D_{2^n}$, then since ${\rm Aut} (D_{2^n})$ is a $2$-group, by the Normalizer-Centralizer Theorem, we conclude that $P\leq \textbf{C}_K(Q)$, and so $C_2^2\times C_p$ is a proper PF-subgroup of $K$, a contradiction. Thus, $Q\cong C_q^2$. Hence, $K$ is isomorphic to the quasi-Frobenius group $(C_q: C_p)\times C_q$, $C_q^2\times C_p$ or a Frobenius group. Since $K$ admits a non-trivial partition,  looking at examples (4) and (7) of Section 3, $K$ must be a Frobenius group. Now, let $H$ be a subgroup of $G$ isomorphic to $H_q(M): P$. Since $K$ is a Frobenius group, $H$ is not cyclic, and so $H$ must be a Frobenius group ($H$ is not isomorphic to $S_4$ and if it is a group of Hughes-Thompson type, Remark 
\ref{smallresults}(ii) yields that $H$ is a Frobenius group), implying that $P$ acts on $M$  Frobeniusly, a contradiction. Thus, $p$ is a divisor of $|M|$.

\bigskip

{\bf Step 3.} $q$ must be equal to $p$.

\bigskip

Suppose that $q\not =p$, $P\in {\rm Syl}_p(G)$ and $P_0\in {\rm Syl}_p(M)$. By Step 1, $H_q(M)$ is cyclic, and so by Lemma \ref{p-groups}, $P\cong D_{2^m}$, for some $m\geqslant 3$,  $P\cong C_p^2$ or $P$ is cyclic. Let $P\cong D_{2^m}$, then $M$ contains a subgroup isomorphic to $ C_q$ acting on the $q$-complement of $H_q(M)$ Frobeniusly, by Remark \ref{smallresults}(i). Hence $M$ has a Frobenius subgroup isomorphic to $P_0:C_q$, where $P_0$ is a cyclic 2-group, a contradiction.

Now, assume that $P\cong C_p^2$. Hence, $\textbf{C}_G(P_0)=PH_q(M)$. Thus, the normality of $P_0$ in $G$ implies that $G/H_q(M)$ contains a normal Sylow $p$-subgroup. In addition, $M/H_q(M)$ is the normal Sylow $q$-subgroup of $G/H_q(M)$, and so $G/H_q(M)$ is a cyclic group of order $pq$. Let $K$ be a Hall $\{p,q\}$-subgroup of $G$. Then, $M_0=K\cap H_q(M)$ is a normal subgroup of $K$ and $K/M_0$ is a cyclic group of order $pq$. Recall that by Step 1,  $Q$ is isomorphic to either  $C_q^2$ or $D_{2^n}$. Then, we have $M_0\cong C_{pq}$ or $C_{2^{n-1}p}$. In both cases, $|{\rm Aut}(M_0)|$ is not divisible by $r=max\{p,q\}$, and so $|K/\textbf{C}_K(M_0)|$ is not divisible by $r$. Therefore, $\textbf{C}_K(M_0)$ contains a subgroup isomorphic to $C_r^2\times C_s$, where $s=min\{p,q\}$, a contradiction.

Hence, $P$ is a cyclic group of order greater than $p$. Since $H_q(M)P<G$ does not admit  a non-trivial partition, we deduce  that $H_q(M)P$ is cyclic (notice that $P\cap H_q(M)$ is not trivial and so $H_q(M)P$ is not a Frobenius group, Hughes Thompson type or isomorphic to $S_4$).  Therefore, by Remark \ref{smallresults}(i), $C_G(H_q(M))=H_q(M)P$. As $M/H_q(M)\unlhd G/H_q(M)$,  $G/H_q(M)$ has a normal Sylow $q$-subgroup and Sylow $p$-subgroup, and so $G/H_q(M)$ is isomorphic to $C_{pq}$. Recall that $Q\cong (Q\cap H_q(M))\rtimes Q_0$, where $Q_0\cong C_q$. If $K$ is a Hall $\{p,q\}$-subgroup of $G$, then $G/H_q(M)\cong K/(K\cap H_q(M))$.  Note also that by Fitting's Lemma (\cite[Theorem 4.34]{isaacs}), $K/(Q\cap H_q(M))\cong P\rtimes Q_0$ is either a cyclic group or a Frobenius group. Since $G/H_q(M)$ is cyclic, we get that $K/(Q\cap H_q(M))$ must be cyclic which is a contradiction since $Q_0$ acts Frobeniusly on $P\cap H_q(M)$. 

\bigskip

{\bf Step 4.} For all $x\in P\setminus M$, we have $o(x)=p$, where $P\in $ Syl$_p(G)$. 

\bigskip  

By Steps 1 and 3,  we get that $|P|\geqslant p^3$, where $P\in {\rm Syl}_p(G)$. Suppose that there exists $x\in P\setminus M$ such that $o(x)>p$. Let $H= H_p(M) \langle x\rangle$. Easily we see that,  $H$ is neither  isomorphic to $S_4$ nor  a Frobenius group. Notice that $o(x)>p$, by Steps 1 and 3, $H_p(M)$ is cyclic.
Thus, $H$ is not a group of Hughes-Thompson type. Therefore, $H$ is  cyclic or $H=G$. In the first case, $|P:P_H|=p$, for $P_H\in {\rm Syl}_p(H)$. Note that $P$ is not cyclic since  $M\not =H_p(M)$. Hence, by Lemma \ref{classification}, $P$ is isomorphic to $D_{2^m}$, for some natural number $m\geqslant 3$.  Therefore,   $p=q=2$.  
Hence, $|G:H|=2$, and  since $M\cong (H\cap M): Q_0$, where $Q_0\cong C_2$, we conclude that $G\cong H: Q_0$. Note that the   $2$-complement of $G$ is a subgroup of $H_2(M)$ and Remark \ref{smallresults}(i) yields that   $Q_0$ acts Frobeniusly on the $2$-complement of $G$. Hence,  we conclude that $G\cong D_{2^mn}$, where $n$ is the order of the $2$-complement of $G$ and so $G$ is a group of Hughes-Thompson type, a contradiction. Thus, $G=H=H_p(M)\langle x\rangle$. Note that $x^p\in M$, and so $M=H_p(M)\langle x^p\rangle$. If $o(x)>p^2$, then $o(x^p)>p$, and so $H_p(M)=M$, a contradiction. Therefore, $o(x)=p^2$,      $G\cong H_p(M):\langle x\rangle$, and so $G$ contains a proper PF-subgroup isomorphic to $C_{p}\times C_{p^2}$ (notice that by Step 1, $p=q$ divides $H_p(M)$), which is a contradiction.
Therefore, for every $x\in P\setminus M$, we have $o(x)=p$.

\bigskip

{\bf Step 5.} Final contradiction (in this case $G$ cannot be  a minimal PF-group). 

\bigskip

By Steps 1 and 3, $H_p(M)$ is cyclic, and so $G$ contains a cyclic  normal $p$-complement $H$ contained in $H_p(M)$.    Hence,  $G= H \cup (\cup_{g\in G}  P^g)$. On the other hand, by Steps 1 and 3, the sylow $p$-subgroup of M is not cyclic, implying that $P^g\cap M$ admits a non-trivial partition. 
Therefore, it remains to cover elements in $P^g\setminus (P^g\cap M)$, for all $g\in G$. By Previous step, all elements in those sets would be covered by cyclic groups of order $p$, and so $G$ admits a non-trivial partition, a contradiction.

\bigskip

\textbf{Case (V)\,} Suppose that $M$ is isomorphic to $S_4$.
\bigskip

Since ${\rm Aut}(M)={\rm Aut}(S_4)\cong S_4$, by the Normalizer-Centralizer Theorem, we get that $G/\textbf{C}_G(M)$ is  isomorphic to $M$. Therefore, $\textbf{C}_G(M)\cong C_p$, and so $G\cong C_p \times  S_4$. Consequently, $G$ contains a proper subgroup isomorphic to $A_4\times C_p$. Since for $p\neq 3$, by Remark \ref{smallresults}(ii) and Remark \ref{smallresults}(iv), $A_4\times C_p$ is a PF-group, we must have $p=3$. But in the case $p=3$, by Corollary \ref{smallresultse}(i), $G$ contains a proper PF-subgroup isomorphic to $C_2^2\times C_3$, a contradiction.
\bigskip




\end{document}